\documentclass[reqno,11pt]{amsart}

\usepackage{amsmath,amsfonts,amssymb,amsthm,epsfig}
\usepackage{mathtools}
\usepackage{mathrsfs}
\usepackage[english]{babel}
\usepackage{tikz}
\usepackage[colorlinks,linkcolor=red,citecolor=red]{hyperref}
\usetikzlibrary{arrows}
\usepackage{comment}
\usepackage{graphicx}

\usepackage{a4wide}
\pgfarrowsdeclare{<<<}{>>>}
{
\arrowsize=0.2pt
\advance\arrowsize by .5\pgflinewidth
\pgfarrowsleftextend{-4\arrowsize-.5\pgflinewidth}
\pgfarrowsrightextend{.5\pgflinewidth}
}
{
\arrowsize=0.2pt
\advance\arrowsize by .5\pgflinewidth
\pgfpathmoveto{\pgfpointorigin}
\pgfpathlineto{\pgfpoint{-1.2mm}{-.8mm}}
\pgfusepathqstroke
\pgfpathmoveto{\pgfpointorigin}
\pgfpathlineto{\pgfpoint{-1.2mm}{.8mm}}
\pgfusepathqstroke
}

 \allowdisplaybreaks
\numberwithin{equation}{section}

\font\script=rsfs10 at 11pt
\def\eps{\varepsilon}
\def\H{{\mbox{\script H}\,\,}}

\def\L{{\mbox{\script L}\,\,}}
\def\E{\mathcal E}
\def\M{\mathcal M}
\def\P{\mathcal P}
\def\R{\mathbb R}
\def\S{\mathbb S}

\def\bal{\begin{aligned}}
\def\eal{\end{aligned}}
\def\proofof#1{\begin{proof}[Proof of #1]}
\def\Chi#1{\hbox{{\large $\chi$}{\Large $_{_{#1}}$}}}

\def\step#1#2{\par\noindent{\underline{\it Step~#1.}}\emph{ #2}\\}

\def\XXint#1#2#3{{\setbox0=\hbox{$#1{#2#3}{\int}$} \vcenter{\vspace{-1pt}\hbox{$#2#3$}}\kern-.5\wd0}}

\def\XXiint#1#2#3{{\setbox0=\hbox{$#1{#2#3}{\iint}$} \vcenter{\vspace{-1pt}\hbox{$#2#3$}}\kern-0.5\wd0}}

\def\spt{{\rm spt}}
\newcommand{\norm}[1]{\left\lVert#1\right\rVert}
\newcommand{\weakstar}{\overset{\ast}{\rightharpoonup}}

\newcommand{\scal}[2]{\langle #1, #2 \rangle}

\DeclareMathOperator{\loc}{loc}
\DeclareMathOperator{\dist}{dist}

\newcommand{\res}{\mathop{\hbox{\vrule height 7pt width .5pt depth 0pt
\vrule height .5pt width 6pt depth 0pt}}\nolimits}
\newcommand{\smallres}{\mathop{\hbox{\vrule height 5pt width .5pt depth 1pt
\vrule height -.5pt width 4pt depth 1pt}}\nolimits}
\newcommand\restr[2]{\ensuremath{\left.#1\right|_{#2}}}

\pgfarrowsdeclare{<<<}{>>>}
{
\arrowsize=0.2pt
\advance\arrowsize by .5\pgflinewidth
\pgfarrowsleftextend{-4\arrowsize-.5\pgflinewidth}
\pgfarrowsrightextend{.5\pgflinewidth}
}
{
\arrowsize=0.2pt
\advance\arrowsize by .5\pgflinewidth
\pgfpathmoveto{\pgfpointorigin}
\pgfpathlineto{\pgfpoint{-1.2mm}{-.8mm}}
\pgfusepathqstroke
\pgfpathmoveto{\pgfpointorigin}
\pgfpathlineto{\pgfpoint{-1.2mm}{.8mm}}
\pgfusepathqstroke
}

\pgfarrowsdeclare{<<<<}{>>>>}
{
\arrowsize=0.2pt
\advance\arrowsize by .5\pgflinewidth
\pgfarrowsleftextend{-4\arrowsize-.5\pgflinewidth}
\pgfarrowsrightextend{.5\pgflinewidth}
}
{
\arrowsize=0.2pt
\advance\arrowsize by .5\pgflinewidth
\pgfpathmoveto{\pgfpointorigin}
\pgfpathlineto{\pgfpoint{-1.8mm}{-1.2mm}}
\pgfusepathqstroke
\pgfpathmoveto{\pgfpointorigin}
\pgfpathlineto{\pgfpoint{-1.8mm}{1.2mm}}
\pgfusepathqstroke
}

\newcounter{mt}

\def\maintheoremdeclaration#1{\stepcounter{mt}\newcounter{#1}\setcounter{#1}{\arabic{mt}}}

\maintheoremdeclaration{existence}
\maintheoremdeclaration{Linfty}
\maintheoremdeclaration{Unique}

\newtheorem{theorem}{Theorem}[section]
\newtheorem{lemma}[theorem]{Lemma}
\newtheorem{prop}[theorem]{Proposition}
\newtheorem{corol}[theorem]{Corollary}
\newtheorem{defin}[theorem]{Definition}
\theoremstyle{remark}
\newtheorem{remark}[theorem]{Remark}

\begin{document}

\title[Minimizing sets for non-local energies]{On the existence of minimizing sets for a weakly-repulsive non-local energy}

\author{D. Carazzato}
\address{Davide Carazzato\\ Scuola Normale Superiore, Pisa (Italy)}
\email{davide.carazzato@sns.it}

\author{A. Pratelli}
\address{Aldo Pratelli\\ Department of Mathematics, University of Pisa (Italy)}
\email{aldo.pratelli@unipi.it}

\author{I. Topaloglu}
\address{Ihsan Topaloglu\\ Virginia Commonwealth University, Richmond (USA)}
\email{iatopaloglu@vcu.edu}

\thanks{This is a post-peer-review, pre-copyedit version of an article published in Pure and Applied Analysis. The final
authenticated version is available online at: \url{https://doi.org/10.2140/paa.2024.6.995}.}

\begin{abstract}
We consider a non-local interaction energy over bounded densities of fixed mass $m$. We prove that under certain regularity assumptions on the interaction kernel these energies admit minimizers given by characteristic functions of sets when $m$ is sufficiently small (or even for every $m$, in particular cases). We show that these assumptions are satisfied by particular interaction kernels in power-law form, and give a certain characterization of minimizing sets.
Finally, following a recent result of Davies, Lim and McCann, we give sufficient conditions on the interaction kernel so that the minimizer of the energy over probability measures is given by Dirac masses concentrated on the vertices of a regular $(N+1)$-gon of side length 1 in $\R^N$.
\end{abstract}

\maketitle

\section{Introduction}

Given a radial, interaction kernel $g\in L^1_{\loc}(\R^N)$, one can consider the corresponding energy in different classes. The first natural class is the one of the sets with finite volume, for which the energy is given by
\begin{equation}\label{ourenergy}
\E(E) = \int_E \int_E g(x-y)\,dx\,dy\,.
\end{equation}
A standard relaxation suggests then to consider $L^1$ functions with values in $[0,1]$ (which will be often called ``densities'') in place of sets, for which the energy is
\begin{equation}\label{energy2}
\E(h) =\int_{\R^N}\int_{\R^N} g(x-y)h(x)h(y)\,dx\,dy\,.
\end{equation}
A further relaxation consists in directly considering finite positive measures $\M^+$, and defining 
\begin{equation}\label{eq:energy}
\E(\mu) = \iint g(x-y)\,d\mu(x)\,d\mu(y)\,.
\end{equation}
The corresponding minimization problems are then, for any given $m>0$,
\begin{gather}
\min\Bigl\{ \E(E) \colon E\subseteq\R^N,\, |E|=m \Bigr\}\,,\label{eq:problem-with-sets}\\
\min\Bigl\{\E(h) \colon h\in L^1(\R^N), \ 0\leq h\leq 1, \ \norm{h}_1=m \Bigr\}\,,\label{eq:problem-with-densities}\\
\min\Bigl\{ \E(\mu) \colon \mu\in\M^+(\R^N),\, \|\mu\|_\M=1\Bigr\}\,.\label{eq:problem-with-measures}
\end{gather}
Of course, each problem extends the previous ones, since for any set $E\subseteq\R^N$ one has $\E(E)=\E(\Chi{E})$ and $|E|=\norm{\Chi{E}}_1$, and similarly for any function $h:\R^N\to [0,1]$ one has $\E(h) = \E(h\L)$ and $\norm h_1=\norm{h\L}_\M$, where $\L$ denotes the Lebesgue measure in $\R^N$. These minimization problems and the relation between them -- often for a specific choice of $g$ -- have been extensively investigated by many authors, with a particular effort in the last decade; some references are, for instance,~\cite{BCT,C,CFuP,CS,CFT,F2M3,FL1,FL2,Lop}. As soon as $g$ is lower semicontinuous and $g(x)\to+\infty$ when $|x|\to+\infty$, the existence of a minimizer in~(\ref{eq:problem-with-measures}) follows from compactness arguments (see~\cite{CCP,CP,SST}).

Let us briefly describe some of the known results in a specific, important case, namely, for the attractive-repulsive kernel given in the power-law form $g_1(x)=|x|^\alpha + \frac 1{|x|^\beta}$, where $\alpha>0$ and $0<\beta<N$. First of all, in~\cite{BCT}, it was shown that a set $E$ is a minimizer for~\eqref{eq:problem-with-sets} if and only if its characteristic function $\Chi E$ is a solution of~\eqref{eq:problem-with-densities}. The same holds true also in a more general setting.

If $N-2\leq \beta < N$, in~\cite{CDM} it is shown that the optimal measures are actually $L^\infty$-functions. Recently, in~\cite{CP}, the first two authors extended this result to a wide class of generic interaction kernels, also providing an a priori bound on the $L^\infty$-norm of minimizers.

In addition, if $0<\beta<N-1$, using quantitative rearrangement inequalities Frank and Lieb proved in~\cite{FL2} the existence of a threshold $m_{\rm ball}\in(0,\infty)$ such that the ball with volume $m$ is the only (up to translation) minimizer of~\eqref{eq:problem-with-sets} for $m>m_{\rm ball}$. In the special case of quadratic attraction ($\alpha=2$), Burchard, Choksi and the third author had proven the same result in~\cite{BCT}, by exploiting the convexity of the energy among densities $h$ with fixed center of mass. Consequently, Lopes used a similar argument in~\cite{Lop} to prove that for $2<\alpha\leq 4$ (and any $0<\beta<N$) minimizers of~\eqref{eq:problem-with-densities} are radially symmetric and unique up to translations. Recently, the first author extended the results of~\cite{FL2} and showed that when the interaction kernel is given by $g(x)=|x|^\alpha + \widetilde{g}(x)$, for a wide class of functions $\widetilde{g}(x)$, the unique (up to translations) minimizers of~\eqref{eq:problem-with-densities} are given by the characteristic function of balls when $m$ is sufficiently large (see~\cite{C}). The stability and local minimality of the ball when $0<\beta<N-1$ has also been studied in~\cite{BCT2}. On the other hand, in the small volume regime, the energy~\eqref{eq:problem-with-sets} does not admit a minimizer for $N>2$, $\alpha=2$, $N-2\leq \beta< N$, and for $N=3$, $\alpha>0$, $\beta=1$, as was shown in~\cite{BCT} and~\cite{FL1} respectively. In these cases, the minimizers of~\eqref{eq:problem-with-densities} actually satisfy $h<1$ almost everywhere.\par

In this paper, we are interested in generic interaction kernels which are \emph{weakly repulsive} (at the origin). This means that $g(0)=0$, and $g$ is negative for small distances and positive for larger ones. In particular, while for strongly repulsive kernels, like $g_1$, a measure containing some atom has always infinite energy, for weakly repulsive kernels atomic measures have finite energy, and hence they are possible candidates for the minimization problem~(\ref{eq:problem-with-measures}). This is not just a theoretical possibility; in fact, Carrillo, Figalli and Patacchini showed in~\cite{CFiP} that global minimizers of~\eqref{eq:energy} over probability measures are supported on finitely many points if $g(x)\approx |x|^\beta$ for some $\beta>2$ when $|x|\ll 1$. Here, writing with a small abuse of notation $g(|x|)=g(x)$, by `` $g(x)\approx |x|^\beta$ '' we mean that $g(0)=0$ and $g'(t) t^{1-\beta}\to -C$ as $t\to 0$ for some $C>0$.

An important example of a weakly repulsive kernel is given by
\begin{equation}\label{defprot}
g_2(x) = \frac{|x|^\alpha}\alpha - \frac{|x|^\beta}\beta,\qquad\qquad \alpha>\beta>0\,.
\end{equation}
Dividing by $\alpha$ and $\beta$ clearly makes no difference, but it is convenient so that the minimal interaction is reached at distance $1$. It is possible to apply arguments by Frank and Lieb to $g_2$ and see that again, when $m$ is large enough, the minimizers of~\eqref{eq:problem-with-densities} are characteristic functions of a ball. This kind of kernel has been studied by Davies, Lim and McCann in a series of papers~(\cite{DLM1,DLM2,LM}), and they are able to precisely characterize the solutions of~(\ref{eq:problem-with-measures}) in some cases.

\begin{theorem}[Davies--Lim--McCann, \cite{DLM1,DLM2}]\label{thmDLM}
Let $N\geq 2$, and $g=g_2$ be given by~(\ref{defprot}). If $\beta=2<\alpha<4$, then the unique minimizer of~(\ref{eq:problem-with-measures}), up to rigid motion, is given by the uniform distribution over a sphere, that is, $\mu = c \H^{N-1}\res \partial B_r$ for a suitable choice of $c$ and $r$. If $\alpha\geq 4,\, \beta\geq 2$ and $(\alpha,\beta)\neq (4,2)$, then the unique minimizer is given by a purely atomic measure uniformly distributed over the vertices of the unit regular $(N+1)$-gon $\Delta_N$.
\end{theorem}

The minimizers have been investigated also in dimension $N=1$. In this case, the unique minimizer is given by two equal masses at distance $1$ as soon as $\alpha\geq 3$, $\beta\geq 2$ (see~\cite{DLM2}), while for $2<\alpha<3$, $\beta=2$ the minimizer, which is computed explicitly in~\cite{Fra}, is absolutely continuous and supported on an interval.\medskip

The goal of this paper is to study the question of existence of optimal \emph{sets}, that is, minimizers of~(\ref{eq:problem-with-sets}). First of all, we underline that existence should not be expected in general. Indeed, as said above, relaxation arguments allow to deduce that a minimizing set exists if and only if there is a function minimizing~(\ref{eq:problem-with-densities}) which is a characteristic function, and this is, of course, a peculiar situation. As a matter of fact, in all the results where existence of optimal sets is established, as in the ones described above, the optimal sets are actually balls, and there is not really an argument which provides existence, but rather the existence is simply obtained as a consequence of the optimality of the balls. It is worth noting that there are, in fact, non-local energies for which existence of optimal sets (different from balls) is known. However, they are not of the form~(\ref{ourenergy}), but of the form
\begin{equation}\label{otherform}
\widetilde \E (E) = P(E) + \int_E\int_E \tilde g(y-x)\,dx\,dy\,,
\end{equation}
where $P(E)$ denotes the perimeter of $E$ and $\tilde g$ is a rather general function, the prototype being negative powers of the distance. For energies of this form, existence of optimal sets has been established under different assumptions, see for instance~\cite{FL0,KM1,KM2,NP}. However, the situation is completely different, due to the presence of the perimeter instead of a non-local double integral. This dramatically changes the situation, because the class of sets is no more dense, in the energy sense, in the class of $L^1$ functions (of course, among functions, the perimeter term has to be replaced by the total variation). Hence, it is not true that existence of optimal sets implies that some optimal $L^1$ function must be a characteristic function. As a consequence, existence of optimal sets is not something to be in general unexpected, on the contrary it usually comes as a rather standard application of compactness in BV, the only difficulty being in excluding loss of mass at infinity.\par

Summarizing, up to now in the literature existence of optimal sets was either more or less simple by standard methods, as for energies like~(\ref{otherform}), or it was only obtained in special cases where the optimal sets are actually balls, for energies like~(\ref{ourenergy}). What we are able to do in this paper is to provide a first argument ensuring the existence of optimal sets for some weakly repulsive kernels. That is, we prove that in some cases when an optimal measure is concentrated on a negligible set, an optimal set exists if the mass is small. Even though the precise statement, Theorem~\ref{thm:existence-sets}, has rather technical assumptions, its meaning becomes particularly evident having in mind Theorem~\ref{thmDLM}. Indeed, when the powers ensure that the optimal measure is uniformly distributed over the vertices of the unit, regular $(N+1)$-gon $\Delta_N$, then we get existence of optimal sets for small mass, and these sets are made by $N+1$ disjoint, convex subsets close to the vertices of $\Delta_N$, see Theorem~\ref{thm:existence-power-like}. Instead, when the powers are so that the optimal measure is uniformly distributed over a sphere, then we get existence of optimal sets for \emph{every} mass, and the solutions are always either annuli or balls, see Theorem~\ref{thm:existence-radial}.\par

We remark that the existence of minimal sets for similar energies was also investigated, with different techniques, by Clark in her Ph.D. thesis~\cite{Cla}.

\subsection{Plan of the paper}

The plan of the paper is the following. In Section~\ref{sec:preliminaries} we introduce some notation and some basic results, in particular the existence of minimizers for~\eqref{eq:problem-with-densities} and the corresponding optimality conditions, and an explicit bound on the diameter of their support, see Proposition~\ref{prop:diameter-bound}.

In Section~\ref{sec:small-volume} we prove our main results. Specifically, Theorem~\ref{thm:existence-sets} provides existence of an optimal set for small mass under some technical assumptions on $g$. Then, in Theorem~\ref{thm:existence-power-like}, we observe that this abstract result can be applied for instance in the cases when Theorem~\ref{thmDLM} ensures that the optimal measure is given by atoms in the vertices of $\Delta_N$. The existence of optimal sets is true also when Theorem~\ref{thmDLM} says that the optimal measure is uniformly distributed over a sphere, and this is the content of Theorem~\ref{thm:existence-radial}, which is valid not only for small mass but \emph{for all} $m$ values. Finally, Theorem~\ref{Thm3d} gives some technical conditions under which the minimal measures are concentrated on the vertices of $\Delta_N$. This generalizes the results by Davies, Lim and McCann. In this case the abstract existence result of Theorem~\ref{thm:existence-sets} can be applied again to show that optimal sets exist for small mass. We highlight that the hypotheses of this last theorem are stable with respect to small perturbations, showing that the existence of minimizing sets is not a feature possessed only by some specific kernels.

We conclude by observing that, in the cases considered by Theorem~\ref{thmDLM}, we prove existence of an optimal set for small mass, and as described above existence is also known for large mass, when the optimal set is a ball. This leaves open the question of what happens in the intermediate volume regimes. On one side, as noted above, one should in general not expect existence of optimal sets. But on the other side, at least for the cases when the optimal measure is uniformly distributed over a sphere, our results provide existence of optimal measures also for intermediate masses.

\section{Notation and preliminary results}\label{sec:preliminaries}

This section is devoted to introduce the few notation that we use and to gather a couple of useful results. Through this paper, $g:\R^N\to\R$ denotes a radial, $L^1_{\rm loc}$, lower semicontinuous function. Since $g$ is radial, with a slight abuse of notation we will often write $g(t)=g(x)$ for any $|x|=t$. We use the letter $\E$ to denote the interaction between two densities. This means that, given $h_1,\, h_2\in L^1(\R^N;[0,1])$, we write
\[
\E(h_1,h_2) = \int_{\R^N}\int_{\R^N} g(x-y)h_1(x)h_2(y)\,dx\,dy\,,
\]
so that in particular, according to the notation~(\ref{energy2}), we have $\E(h)=\E(h,h)$. The very same notation is used for the case of two sets, or two measures, extending~\eqref{ourenergy} and~\eqref{eq:energy} respectively. The \emph{potential} of a measure $\mu$ is the function $\psi_\mu:\R^N\to\R$ defined as
\[
\psi_\mu(x) = \int g(x-y) \,d\mu(y)\,,
\]
and given the density $h$ we denote for brevity by $\psi_h$ the potential of the measure $h\L$. The potential is very useful in computing the energy of a measure, in particular of course $\E(\mu) = \int \psi_{\mu}\,d\mu$. Moreover, it naturally appears in the Euler--Lagrange conditions for the minimization problem, thanks to the following standard result (we give a short sketch of the proof, the formal one can be seen for instance in~\cite{CDM} or~\cite{BCT}).
\begin{prop}\label{prop:EL}
Let $g\in L^1_{\loc}(\R^N)$ be a function bounded from below. If $f$ and $\mu$ are minimizers of~(\ref{eq:problem-with-densities}) and~(\ref{eq:problem-with-measures}) respectively, then
\begin{align}\label{eq:EL}\tag{$EL$}
\begin{cases}
\psi_f = \lambda \qquad \L\text{-a.e. in }\{0<f<1\},\\
\psi_f \geq \lambda \qquad \L\text{-a.e. in }\{f=0\},\\
\psi_f \leq \lambda \qquad \L\text{-a.e. in }\{f=1\},
\end{cases}
 &&
\begin{cases}
\psi_{\mu} = \E(\mu)\qquad \mu\text{-a.e.},\\
\psi_{\mu} \geq \E(\mu) \qquad \text{in }\R^N\setminus\spt\mu,
\end{cases}
\end{align}
for some constant $\lambda\in(-\infty,+\infty]$.
\end{prop}
\begin{proof}[Sketch of the proof]
We concentrate only on the case of densities since it is the most relevant one in this paper, the other case is completely similar. Given any function $h:\R^N\to[-1,1]$ with compact support and satisfying $\int h=0$ and $0\leq f+h\leq 1$, for any $t\in[0,1]$ we have
\[
\E(f+th) = \E(f)+2t\int \psi_f(x)h(x)\,dx+t^2\E(h)\,.
\]
For any $0\leq t \leq 1$ the function $f+th$ is an admissible competitor in~\eqref{eq:problem-with-densities}, thus by minimality we have $\E(f)\leq \E(f+th)$. Then, also using that $\E(h)<+\infty$ since $\spt h$ is compact, we deduce that $\psi_f(x)\leq \psi_f(y)$ for any two points $x,\,y$ such that $f(x)>0$ and $f(y)<1$. This is stronger than~\eqref{eq:EL}.
\end{proof}

Another standard result is the existence of minimizers for problems~(\ref{eq:problem-with-densities}) and~(\ref{eq:problem-with-measures}). The proof in our general setting can be easily adapted from those already available in the literature, see for instance~\cite{CCP,C,SST}.

\begin{lemma}\label{lemma:existence}
Assume that $g\in L^1_{\loc}(\R^N)$ is bounded from below, lower semicontinuous\,and $\lim_{|x|\to+\infty}g(x)=+\infty$. Then, for any $m>0$ there exist a minimizer of~\eqref{eq:problem-with-densities} and a minimizer of~\eqref{eq:problem-with-measures}.
\end{lemma}

The last result that we present is an a-priori bound on the diameter of the support of a minimizing density, and this deserves a quick comment. When dealing with minimizing measures, the boundedness of the support is a quite standard result, and it has been proved in several different contexts (see for instance~\cite{CCP,CP}). As we have already noticed, for many properties (for instance the existence given by the above lemma) working with measures or with densities does not make much difference. However, the compactness of the support of minimizers is more delicate for the case of densities due to the fact that the Euler--Lagrange condition~(\ref{eq:EL}) for densities has an additional constraint (see \cite{FL1} for the special case when $g$ is given by a power-law of the form $g_1$). As a consequence, the proof of the result below does not follow by a simple generalization of the proofs available for the case of measures. Therefore we provide a complete proof.

\begin{defin}
A radial function $g:\R^N\to\R$ is said \emph{definitively non-decreasing} if there exists some $\overline R\geq 0$ such that $g(s)\geq g(t)$ for every $s\geq t\geq \overline R$.
\end{defin}

\begin{prop}\label{prop:diameter-bound}
Let $g\in L^1_{\rm loc}(R^N)$ be a radial, lower semicontinuous, bounded from below, definitively non-decreasing function such that $\lim_{|x|\to+\infty}g(x)=+\infty$. Then, there is a constant $D=D(g,m)$ such that the diameter of the support of any density minimizing~(\ref{eq:problem-with-densities}) is bounded by $D$. If, in addition, $g$ is locally bounded, then one can take the same constant $D(g,m)$ for every $m\leq 1$.
\end{prop}
\begin{proof}
We can assume without loss of generality that $g\geq 0$. Let $f:\R^N\to [0,1]$ be any minimizer of~(\ref{eq:problem-with-densities}). In particular, $\E(f)\leq C_m$, where $C_m=C_m(g,m)$ is the energy of a ball of volume $m$. Then, let us fix a constant $R=R(g,m)$ so that $g$ is non-decreasing on $[R,+\infty)$, the volume of the ball of radius $R$ is bigger than $\kappa m$, where $\kappa=\kappa(N)$ is a purely geometric constant that will be defined later, and
\begin{equation}\label{Rbig}
g(R)> \frac{5C_m}{m^2}\,.
\end{equation}
Let us now call $f_1 = f \Chi{B_R}$ and $\delta=m - \|f_1\|_{L^1}$. Up to a translation, we can assume that $\delta \leq m/5$. Indeed, if this was not true, then for any $x\in\R^N$ the mass of $f$ outside of the ball $B(x,R)$ would be more than $m/5$, and then also by~(\ref{Rbig})
\[\begin{split}
\E(f) &\geq \int_{\R^N} \int_{\R^N\setminus B(x,R)} g(y-x)f(y)f(x)\,dy\,dx
\geq \int_{\R^N} g(R) \bigg( \int_{\R^N\setminus B(x,R)} f(y)\,dy\bigg) f(x)\,dx\\
&\geq \int_{\R^N} \frac m5\, g(R) f(x)\,dx = \frac {m^2}5\, g(R) > C_m\,,
\end{split}\]
against the optimality of $f$. We let now $R^+=R^+(g,m)\geq 50R$ be another constant, so that
\begin{equation}\label{defR+}
g(R^+-R)\geq 2 g(6R) +\frac 5{2m}\, \int_{B_{11R}} g(x)\,dx\,,
\end{equation}
and we aim to prove that $f$ is supported in $B_{R^+}$, so that the proof will be concluded with $D=2R^+$. Let us call $f_2 = f\Chi{B_{R^+}\setminus B_R}$ and $f_3 = f\Chi{\R^N\setminus B_{R^+}}$, so that $f=f_1+f_2+f_3$. Calling now $\eps=\|f_3\|_{L^1}\leq \delta\leq m/5$, our claim can be rewritten as $\eps=0$, thus we assume $\eps>0$ and we look for a contradiction.

Let $z^+$ be a minimizer of the potential $\psi_{f_2}(z)= \int_{\R^N} g(z-y)\, f_2(y)\,dy$ within the support of $f_3$. Notice that such a minimizer exists. Indeed, by assumption the support of $f_3$ is a non-empty closed set, and the above function is either constantly $0$ if $f_2\equiv 0$ (and in such a case any point of the support is a minimizer), or it is a continuous function which explodes for $|z|\to \infty$. The minimality property of $z^+$ ensures that
\begin{equation}\label{firstone}
\E(f_2,f_3) = \int_{\R^N} \psi_{f_2}(z) f_3(z)\, dz \geq \psi_{f_2}(z^+) \|f_3\|_{L^1}=\psi_{f_2}(z^+) \eps\,.
\end{equation}
Let us now define the set
\[
\mathcal C = \bigg\{z\in\R^N:\, 4R\leq |z|\leq 5R, \, \frac{z\cdot z^+}{|z|\cdot |z^+|} \geq \cos(\pi/15) \bigg\}\,,
\]
which is the portion of cone highlighted in Figure~\ref{Fig1}. We call then $\kappa=\omega_N R^N/|\mathcal C|$, which is a purely geometrical constant only depending on $N$. Then, since by construction $|B_R|> \kappa m$, we have $|\mathcal C| = |B_R| /\kappa >m$.
\begin{figure}[htbp]
\begin{tikzpicture}[scale=0.3]
\draw (0,0) circle (2.5);
\draw (-2.2,2) node[anchor=east] {$B_R$};
\fill (1,1.5) circle(4pt);
\draw (1,1.5) node[anchor=north east] {$w$};
\draw (24,0) arc(0:15:20);
\draw (24,0) arc(0:-15:20);
\draw (23.7,3) node[anchor=west] {$B_{R^+}$};
\fill[blue!25!white] (5.5,0) arc(0:12:5.5) -- (7.83,1.66) arc(12:-12:8) -- (5.38,-1.14);
\draw (5.5,0) arc(0:12:5.5) -- (7.83,1.66) arc(12:-12:8) -- (5.38,-1.14);
\draw (5.5,0) arc(0:50:5.5);
\draw (5.5,0) arc(0:-50:5.5);
\draw (8,0) arc(0:40:8);
\draw (8,0) arc(0:-40:8);
\draw (7,3.7) node[anchor=west] {$B_{5R}$};
\draw (3.6,4.5) node[anchor=west] {$B_{4R}$};
\fill (24.5,0) circle(4pt);
\draw (24.5,0) node[anchor=north west] {$z^+$};
\draw[dashed] (12,-5)--(12,5);
\fill (13,1.5) circle(4pt);
\draw (13,1.5) node[anchor=north west] {$y'''$};
\fill (10,-3) circle(4pt);
\draw (10,-3) node[anchor=north west] {$y''$};
\end{tikzpicture}
\caption{The construction in Proposition~\ref{prop:diameter-bound}.}\label{Fig1}
\end{figure}
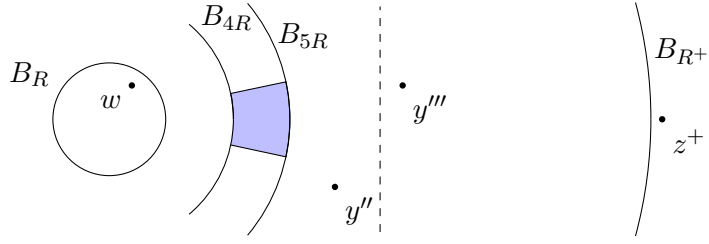
 Since $\|f\|_{L^1(B_{R^+})}=m-\eps$, there exists a positive density $\tilde f_3$, concentrated in $\mathcal C$, such that 
\begin{align*}
\|\tilde f_3\|_{L^1}=\eps\,, && 0\leq \tilde f :=f_1+f_2+\tilde f_3\leq 1\,.
\end{align*}
In particular, the fact that $\tilde f_3$ is concentrated in $\mathcal C$ gives
\begin{equation}\label{almostpar}
\frac{z \cdot z^+}{|z|\cdot |z^+|} \geq \cos(\pi/15) \qquad \forall\, z\in \spt(\tilde f_3)\,.
\end{equation}
We will conclude our proof by showing that $\E(f)>\E(\tilde f)$, which will contradict the minimality of $f$ since by construction $\tilde f$ is a competitor for problem~(\ref{eq:problem-with-densities}). Notice that
\begin{equation}\label{evaldiff}
\E(f) - \E(\tilde f) = 2 \Big(\E(f_1,f_3) - \E(f_1,\tilde f_3)+\E(f_2,f_3) - \E(f_2,\tilde f_3)\Big) + \E(f_3)-\E(\tilde f_3)\,.
\end{equation}
Let us evaluate separately the different pieces. First of all, by construction
\begin{gather*}
\E(f_1,f_3) \geq g(R^+-R) \|f_3\|_{L^1} \|f_1\|_{L^1} = g(R^+-R) \eps (m-\delta)\,,\\
\E(f_1,\tilde f_3) \leq g(6R) \|\tilde f_3\|_{L^1} \|f_1\|_{L^1} = g(6R) \eps (m-\delta)\,,
\end{gather*}
thus by~(\ref{defR+}) and since $\delta\leq m/5$ and~(\ref{Rbig}) we have
\begin{equation}\label{insert1}
\E(f_1,f_3) -\E(f_1,\tilde f_3) \geq \frac 45\, m \eps \bigg( g(6R)+ \frac 5{2m}\, \int_{B_{11R}} g(x)\,dx \bigg)> 4 \eps\,\frac{C_m}m + 2\eps \int_{B_{11R}} g(x)\,dx\,.
\end{equation}
To estimate $\E(f_2,f_3)-\E(f_2,\tilde f_3)$, it is convenient to subdivide $\R^N$ into three pieces. The first one is the ball $H'=B_{6R}$, and the other two are
\begin{align*}
H''= \bigg\{ x \notin H' :\, \frac{x\cdot z^+}{|z^+|} \leq \frac 12 \, R^+\bigg\}\,, && H'''= \bigg\{ x \notin H' :\, \frac{x\cdot z^+}{|z^+|} > \frac 12 \, R^+\bigg\}\,,
\end{align*}
which are respectively on the left and on the right of the dashed hyperplane in the figure. We call then $f_2',\, f_2''$ and $f_2'''$ the restrictions of $f_2$ to $H',\, H''$ and $H'''$, so that $f_2=f_2'+f_2''+f_2'''$. We now observe that
\begin{equation}\label{insert2}
\begin{split}
\E(f_2',\tilde f_3) &= \int_{\R^N}\int_{H'} g(y'-z)\, \tilde f_3(z) f_2(y')\,dy'\,dz
\leq \int_{\R^N} \int_{B_{6R}} g(y'-z) \tilde f_3(z) \,dy'\, dz\\
&\leq \int_{\R^N} \int_{B_{11R}} g(x) \tilde f_3(z) \,dx\, dz = \eps \int_{B_{11R}} g(x)\,dx\,.
\end{split}\end{equation}
Next, we pass to $f_2''$. For any $y''\in H''\cap B_{R^+}$ and $z\in \spt(\tilde f_3)$, by construction and using~(\ref{almostpar}) we have $R< |y''-z|\leq |y''-z^+|$. Since $g$ is non-decreasing on $[R,+\infty)$, also by~(\ref{firstone}) we can evaluate
\begin{equation}\label{insert3}
\begin{split}
\E(f_2'',\tilde f_3) &= \int_{H''}\int_{\R^N} g(y''-z) f_2(y'') \tilde f_3(z) \, dz\,dy''\\
&\leq \int_{\R^N}\int_{\R^N} g(y''-z^+) f_2(y'') \tilde f_3(z) \, dz\,dy''
= \eps \int_{\R^N} g(y''-z^+) f_2(y'') \,dy''\\
&=\eps \psi_{f_2}(z^+) \leq \E(f_2,f_3)\,.
\end{split}\end{equation}
The argument to estimate $\E(f_2''',\tilde f_3)$ is similar. Since for any $w\in B_R$, any $y'''\in H'''\cap B_{R^+}$, and any $z$ in the support of $\tilde f_3$ we have, by construction and an elementary trigonometric calculation, $|y'''-w|\geq |y'''-z|>R$, we evaluate
\[
\E(f_2''',\tilde f_3) = \int_{H'''}\int_{\R^N} g(y'''-z) f_2(y''')\tilde f_3(z)\,dz\,dy'''
\leq \eps\int_{H'''} g(y'''-w) f_2(y''')\,dy'''\,.
\]
Since this is true for every $w\in B_R$, and $f_1$ is concentrated on $B_R$, we obtain
\[\begin{split}
(m-\delta) \E(f_2''',\tilde f_3)&=\int_{B_R} \E(f_2''',\tilde f_3)f_1(w)\,dw
\leq \eps \int_{B_R} \int_{H'''} g(y'''-w) f_2(y''') f_1(w)\,dy'''\,dw\\
&= \eps \,\E(f_1,f_2''') \leq \eps \E(f) \leq \eps\, C_m\,,
\end{split}\]
which since $\delta\leq m/5$ implies
\begin{equation}\label{insert4}
\E(f_2''',\tilde f_3) \leq 2\eps\, \frac{C_m}m\,.
\end{equation}
Putting together~(\ref{insert2}), (\ref{insert3}) and~(\ref{insert4}), we have
\begin{equation}\label{insert5}
\E(f_2,f_3)-\E(f_2,\tilde f_3)> -2\eps\, \frac{C_m}m-\eps \int_{B_{11R}} g(x)\,dx
\end{equation}
Lastly, since the support of $\tilde f_3$ is contained in $\mathcal C$, whose diameter is much smaller than $11R$, we can readily estimate
\begin{equation}\label{insert6}
\E(\tilde f_3)=\E(\tilde f_3,\tilde f_3) \leq \int_{\R^N} \int_{B_{11R}(z)}g(z-y)\tilde f_3(z)\, dy\, dz =\eps \int_{B_{11R}} g(x)\,dx\,.
\end{equation}
Inserting~(\ref{insert1}), (\ref{insert5}) and~(\ref{insert6}) into~(\ref{evaldiff}), and minding also $\E(f_3)\geq 0$, we finally obtain
\[
\E(f) - \E(\tilde f) \geq 4 \eps\,\frac{C_m}m + \eps \int_{B_{11R}} g(x)\,dx>0 \,,
\]
thus the contradiction $\E(f)>\E(\tilde f)$ is established and this concludes the first part of the proof.\par

Assume now that $g$ is locally bounded, and let us notice that a simple modification of the proof provides the same constant $D(g,m)$ for every $m\leq 1$. Notice first that any ball with volume $m\leq 1$ has diameter less than $\omega_N^{-1/N}$, thus $C_m \leq C m^2$, where $C=\sup \{g(x),\, |x|\leq 2\omega_N^{-1/N}\}$. As a consequence, one can take the same radius $R$ in~(\ref{Rbig}) for every $m\leq 1$. The radius $R^+$ defined in~(\ref{defR+}) explodes when $m\searrow 0$, but the local boundedness of $g$ allows for a much simpler definition of $R^+$. More precisely, (\ref{insert2}) can be clearly modified by saying
\[
\E(f_2',\tilde f_3)\leq \frac{\eps m}5 \,g(11R)\,,
\]
and then the proof works with no other modification replacing the definition~(\ref{defR+}) of $R^+$ by
\[
g(R^+-R)\geq 2 g(6R) + g(11R)\,,
\]
which does not depend on $m$. Since $D(g,m)=2R^+$, the proof is concluded.
\end{proof}

\section{Existence of minimal sets with small volume} \label{sec:small-volume}

\subsection{Existence results for general kernels}

A minimizer of the problem~\eqref{eq:problem-with-densities} exists under mild hypotheses on $g$, as we recalled in Lemma~\ref{lemma:existence} (and this can be also found in~\cite{C,CFT}), but, in general, this problem does not necessarily admit a minimizer if considered among the restricted class of the characteristic functions (see e.g.~\cite{BCT,FL1} and~\cite[Remark 3.16]{CP}). Here we show that, under certain rather general conditions on $g$, the solutions to the problem~\eqref{eq:problem-with-densities} are characteristic functions of sets also when $m$ is small enough. First we prove a general statement, and then we specialize further connecting our results to those present in~\cite{DLM1,DLM2}.

\begin{lemma}\label{lemma:convergence}
Let $g\in C(\R^N)$ be definitively non-decreasing such that $\lim_{|x|\to+\infty}g(x)=+\infty$. Let $f_j$ be a minimizer of~\eqref{eq:problem-with-densities} with $\norm{f_j}_1=m_j$ for any sequence $m_j\searrow 0$. Then, up to translations and up to taking a subsequence, $m_j^{-1}f_j\weakstar \mu$ for some $\mu\in\P(\R^N)$ minimizing~(\ref{eq:problem-with-measures}). Moreover, if $\psi_{\mu}>\E(\mu)$ in $\R^N\setminus\spt\, \mu$, then for any $\eps>0$ there is $\bar j$ such that
\begin{equation}\label{incluconv}
\spt f_j\subseteq B_\eps+\spt\, \mu\qquad \forall \, j>\bar j\,.
\end{equation}
\end{lemma}
\begin{proof}
Since $g$ is continuous, thus locally bounded, Proposition~\ref{prop:diameter-bound} ensures that the supports of the densities $f_j$ are uniformly bounded. Therefore, the probability measures $\mu_j=m_j^{-1} f_j$ have uniformly bounded support and then, up to subsequences and translations, we have $\mu_j \weakstar \mu$ for some probability measure $\mu$ with bounded support.

Let now $\bar\mu$ be any minimizer of~(\ref{eq:problem-with-measures}), and for any $j$ let $\mathscr{P}_j$ be a partition of $\R^N$ made by pairwise disjoint cubes of volume $m_j$, and define the function $\tilde f_j$ as
\[
\tilde f_j(x) = \bar\mu(Q)\qquad \forall\, Q\in\mathscr{P}_j, \forall\, x\in Q\,.
\]
By construction, $0\leq \tilde f_j\leq 1$ and that $\|\tilde f_j\|_1=m_j$, so by the minimality of $f_j$ we have $\E(f_j)\leq \E(\tilde f_j)$. The continuity of $g$ easily guarantees that $\E(m_j^{-1} \tilde f_j) \to \E(\bar\mu)$, and then also by the lower semicontinuity of $\E$ we deduce
\[
\E(\mu) \leq \liminf \E(\mu_j) = \liminf m_j^{-2}\E(f_j) 
\leq \liminf m_j^{-2}\E(\tilde f_j)
=\liminf \E(m_j^{-1} \tilde f_j)= \E(\bar\mu)\,.
\]
Hence, $\mu$ is a minimizer of $\E$ in $\P(\R^N)$.\par

Suppose now that $\psi_{\mu}>\E(\mu)$ outside of $\spt\mu$, and keep in mind that $\psi_\mu=\E(\mu)$ on $\spt\mu$ by Proposition~\ref{prop:EL}. Since $g$ is continuous and explodes at infinity, and since $\mu$ has bounded support, we deduce that the potential $\psi_\mu$ is continuous and explodes at infinity, thus for any $\eps>0$ there exists $\gamma>0$ such that $\psi_{\mu}(x)\geq \E(\mu)+\gamma$ whenever $\dist(x,\spt\mu)\geq \eps$. Let us now call $U=\spt\mu+B_{\eps}$ and $V = \spt\mu+B_\delta$, with $\delta$ so small that $\psi_{\mu}(x)<\psi_{\mu}(y)-\gamma/2$ for any $x\in V$ and any $y\in U^c$. Since $g$ is continuous and $\spt \mu_j$ are uniformly bounded, then $\psi_{\mu_j}$ are locally uniformly continuous with a common modulus of continuity. The convergence $\mu_j\weakstar \mu$ guarantees then that $\psi_{\mu_j}$ converge pointwise to $\psi_{\mu}$, and thanks to the common modulus of continuity this convergence is locally uniform. Therefore, if $j$ is large enough we have that
\begin{equation}\label{estigamma3}
\psi_{\mu_j}(x)<\psi_{\mu_j}(y)-\frac{\gamma}{3}\qquad\forall x\in V,\ y\in U^c\,.
\end{equation}
Suppose now by contradiction that~(\ref{incluconv}) does not hold true, thus that for some arbitrarily large $j$ the function $f_j$ is not concentrated on $U$. Then, for every $\eta \ll 1$, we can define a modified function $0\leq \hat f\leq 1$ by ``moving a mass $\eta$ from $U^c$ to $V$''. Formally speaking, $\hat f$ is a function such that $0 \leq \hat f\leq f_j$ on $U^c$ while $f_j\leq \hat f\leq 1$ on $V$, and so that
\[
\int_V \hat f - f_j = \int_{U^c} f_j - \hat f = \eta\,.
\]
The existence of such a function $\hat f$ is obvious as soon as $|V|\leq m_j$, which is certainly true for $j$ large enough. We can then call $\hat \mu=m_j^{-1}\hat f$, and $\nu = \mu_j - \hat \mu= m_j^{-1}(f_j-\hat f)$, so that $\|\nu\| = 2m_j^{-1}\eta$. So, we estimate
\[
\E(\hat\mu)-\E(\mu_j) = \E(\nu)+2\E(\mu_j,\nu)
=\E(\nu)+2\int_{\R^N} \psi_{\mu_j}(x)d\nu(x)
\leq C \|\nu\|^2 - \frac \gamma 3\, \|\nu\|\,,
\]
where we have used~(\ref{estigamma3}) and the fact that $g$ is continuous and the support of $\nu$ is bounded. For $\eta\ll 1$ this gives $\E(\hat \mu)<\E(\mu_j)$, thus $\E(\hat f)< \E(f_j)$, and this is impossible since $f_j$ is a minimizer of~\eqref{eq:problem-with-densities} and $\hat f$ is a competitor.
\end{proof}

\begin{theorem}\label{thm:existence-sets}
Let $g\in C^{2}(\R^N)$ be a definitively non-decreasing function such that $\lim_{|x|\to+\infty}g(x)=+\infty$. Let $f_j$ be a minimizer of~\eqref{eq:problem-with-densities} with $\norm{f_j}_1=m_j$ and any sequence $m_j\searrow 0$, and assume that $m_j^{-1}f_j\weakstar \mu\in\P(\R^N)$. If $\psi_{\mu}>\E(\mu)$ in $\R^N\setminus\spt \mu$ and for any $x\in\spt \mu$ there exists $v\in\S^{N-1}$ such that $\partial^{2}_v \psi_{\mu}(x)>0$, then $f_j$ is the characteristic function of a set when $j$ is large enough.
\end{theorem}
\begin{proof}
By Lemma~\ref{lemma:convergence} we know that $\mu$ is a minimizer of~(\ref{eq:problem-with-measures}), and that~(\ref{incluconv}) holds. By assumption, for any $x\in\spt\mu$ there exists some $v_x\in\S^{N-1}$ such that $\partial^2_{v_x} \psi_\mu(x)>0$, and since $\partial^2 \psi_\mu$ is continuous (because $g$ is of class $C^2$ and $\mu$ has bounded support) there exist some $\delta_x,\, r_x>0$ such that $\partial_{v_x}^2 \psi_\mu(y)>2\delta_x$ for every $y\in B_{r_x}(x)$. By compactness, there are finitely many points $x_1,\, x_2,\,\ldots ,\, x_k \in \spt\mu$, corresponding directions $v_1,\, v_2,\, \ldots,\, v_k\in \S^{N-1}$, and two constants $\delta,\, r>0$ such that the balls $B_r(x_i)$ cover the whole $\spt\mu$, and one has
\begin{equation}\label{eq:positive-direction}
\partial^2_{v_i}\psi_{\mu}(y)>2\delta\qquad \forall\, i\in\{1,\ldots,k\}, \forall \,y\in B_r(x_i)\,.
\end{equation}
Since $\spt\mu$ is covered by the finitely many balls $B_r(x_i)$, there exists some $\eps>0$ such that the balls cover also $\spt\mu+ B_\eps$, thus also $\spt f_j$ for any $j$ large enough, by~(\ref{incluconv}). Moreover, since $g\in C^2(\R^N)$ and $m_j^{-1} f_j \weakstar \mu$, we have that $m_j^{-1} D^2 \psi_{f_j}$ converges to $D^2 \psi_\mu$ locally uniformly. Therefore, (\ref{eq:positive-direction}) holds also replacing $2\delta$ with $\delta$ and $\psi_\mu$ with $\psi_{f_j}$ for every $j$ large enough. This condition clearly implies that each level set of $\psi_{f_j}$ has zero measure. But the Euler-Lagrange condition~(\ref{eq:EL}) ensures that $\{0<f_j<1\}$ is contained in a single level set. We deduce then that the function $f_j$ has value $0$ or $1$ almost everywhere, thus it is a characteristic function.
\end{proof}

\begin{corol}\label{cor:existence-sets}
Let $g\in C^{2}(\R^N)$ be a definitively non-decreasing function such that $\lim_{|x|\to+\infty}g(x)=+\infty$. Suppose that, for any minimizer $\mu$ of $\E$ in $\P(\R^N)$, we have
\begin{enumerate}
\item $\psi_{\mu}>\E(\mu)$ in $\R^N\setminus\spt \mu$;
\item for any $x\in\spt \mu$ there exists $v\in\S^{N-1}$ such that $\partial^{2}_v \psi_{\mu}(x)>0$.
\end{enumerate}
Then, there exists $\overline m>0$ such that any $f_m$ minimizing~\eqref{eq:problem-with-densities} with $\norm{f_m}_1=m$ is the characteristic function of a set when $m< \overline m$.
\end{corol}
\begin{proof}
We proceed by contradiction. If the thesis is false, there exists some sequence $m_j\searrow 0$ and densities $f_j$ which minimize~\eqref{eq:problem-with-densities} with mass $m_j$ and are not characteristic functions. Since, as already noticed in the proof of Lemma~\ref{lemma:convergence}, their supports are uniformly bounded, up to a translation and a subsequence we have that $m_j^{-1} f_j\weakstar \mu\in\P(\R^N)$. Since $\mu$ is a minimizer of $\E$ in $\P(\R^N)$ by Lemma~\ref{lemma:convergence}, our assumption ensures that we can apply Theorem~\ref{thm:existence-sets}, clearly obtaining a contradiction.
\end{proof}

\begin{remark}\label{rem:higher-derivatives}
We observe that the proofs of Theorem~\ref{thm:existence-sets} and Corollary~\ref{cor:existence-sets} work also if we have a function $g\in C^{2k}(\R^N)$ and, for a given $\mu$ that minimizes $\E$ in $\P(\R^N)$ (or any minimizer, in the corollary), we have that for every $x\in\spt\mu$ there exist $v\in\S^{N-1}$ and $j\in\{1,\ldots,k\}$ with $\partial_v^{2j}\psi_{\mu}(x)>0$.
\end{remark}

\bigskip

\subsection{More precise results for power-law kernels}

This section is devoted to discuss the situation in the special case of a function $g$ given by~(\ref{defprot}). Let us start with a couple of definitions. We define by $\Delta_N\coloneqq \{x_1,\ldots,x_{N+1}\}\subset\R^N$ the vertices of the standard regular $(N+1)$-gon centered at the origin and with mutual distance $1$. We call $H_N=\sqrt{\frac{N+1}{2N}}$ its height, and $C_N=\sqrt{\frac{N}{2N+2}}$ its circumradius. Moreover, we define
\begin{equation}\label{eq:simplex}
\mu_N \coloneqq \frac{1}{N+1}\sum_{i=1}^{N+1}\delta_{x_i}
\end{equation}
the probability measure which is uniformly distributed over the points of $\Delta_N$. We present now a geometric result which will provide us a positive bound on the second derivative of the potential.

\begin{lemma}\label{lemma:lower-bound-second-derivative}
The constant
\begin{equation}\label{eq:lower-bound-second-derivative}
K_N\coloneqq \min \left\{\sum_{i=1}^{N+1}\scal{v}{x_i-x_1}^2:v\in\S^{N-1}\right\}
\end{equation}
satisfies $K_N=1$ if $N=1$ and $K_N=1/2$ if $N\geq 2$.
\end{lemma}

\begin{proof}
First of all, we claim that for every $N\geq 2$ and every $v\in\S^{N-1}$
\begin{equation}\label{casewithout}
\sum_{i=1}^{N+1} \scal v{x_i}^2 = \frac 12\,.
\end{equation}
To do so, we decompose $v=v_1+v_2$, where $v_2$ is the projection of $v$ onto the hyperplane $\Pi$ parallel to the face containing $x_2,\ldots,x_{N+1}$ and passing through the origin. We can write
\[
\sum_{i=1}^{N+1} \scal v{x_i}^2 = \sum_{i=1}^{N+1} \big(\scal {v_1}{x_i}+\scal{v_2}{x_i} \big)^2
=\sum_{i=1}^{N+1} \scal {v_1}{x_i}^2+\sum_{i=1}^{N+1} \scal{v_2}{x_i}^2 + 2\sum_{i=1}^{N+1} \scal{v_1}{x_i}\scal{v_2}{x_i}\,.
\]
Notice now that by definition $\scal{v_2}{x_1}=0$, and $\scal{v_1}{x_i}$ has the same value for each $i\geq 2$. Since $\sum_{i=1}^{N+1} x_i=0$, we deduce that the last sum vanishes. Moreover, notice that $|x_i|=C_N$, and the distance of any $x_i$ with $i\geq 2$ from the hyperplane $\Pi$ is $H_N-C_N$. Therefore 
\[\begin{split}
\sum_{i=1}^{N+1} \scal v{x_i}^2 &= |v_1|^2 C_N^2 +\sum_{i=2}^{N+1} \scal {v_1}{x_i}^2+\sum_{i=2}^{N+1} \scal{v_2}{x_i}^2 \\
&= |v_1|^2 \Big( C_N^2 + N \big(H_N-C_N\big)^2\Big) + \big( 1 - |v_1|^2\big)\sum_{i=2}^{N+1} \scal{\frac{v_2}{|v_2|}}{x_i}^2 \\
&= \frac {|v_1|^2 }2 + \big( 1 - |v_1|^2\big)\sum_{i=2}^{N+1} \scal{\frac{v_2}{|v_2|}}{x_i}^2 \,.
\end{split}\]
The last expression is linear with respect to $|v_1|^2$. Therefore, either it is constant, or it is minimized for $|v_1|=0$ or $|v_1|=1$. This means that, if the sum in~(\ref{casewithout}) is not constant, then it is minimized only if $v$ is either parallel or orthogonal to $x_1$. However, the same should be true also with any other $x_i$, and this is clearly impossible. We deduce then that the sum in~(\ref{casewithout}) is constant, and then it is enough to choose $|v_1|=1$ to deduce that the constant value is $1/2$, that is, (\ref{casewithout}) is proved.

Let us now consider the sum in~(\ref{eq:lower-bound-second-derivative}). We can assume that $N\geq 3$, since the cases $N=1,\,2$ are elementary computations. Arguing similarly as before, we get
\[
\begin{split}
\sum_{i=1}^{N+1}\scal{v}{x_i-x_1}^2 &= \sum_{i=2}^{N+1}\scal{v_1}{x_i-x_1}^2+\scal{v_2}{x_i-x_1}^2+2\scal{v_1}{x_i-x_1}\scal{v_2}{x_i-x_1}\\
&\hspace{-30pt}=NH_N^2|v_1|^2+\sum_{i=2}^{N+1}\scal{v_2}{x_i}^2-2H_N\scal{v_1}{\frac{x_1}{|x_1|}}\sum_{i=2}^{N+1}\scal{v_2}{x_i}\\
&\hspace{-30pt}=\frac{N+1}2\, |v_1|^2+\sum_{i=2}^{N+1}\scal{v_2}{x_i}^2+2H_N\scal{v_1}{\frac{x_1}{|x_1|}}\scal{v_2}{x_1}
=\frac{N+1}2\, |v_1|^2+\sum_{i=2}^{N+1}\scal{v_2}{x_i}^2\\
&\hspace{-30pt}=\frac{N+1}2\, |v_1|^2+\big(1-|v_1|^2\big)\sum_{i=2}^{N+1}\scal{\frac{v_2}{|v_2|}}{x_i}^2\,.
\end{split}
\]
Notice now that the projections of the points $x_i$ with $2\leq i\leq N$ on the $(N-1)$-dimensional hyperplane $\Pi$ are the vertices of the standard regular $N$-gon centered in the origin. Therefore, the property~(\ref{casewithout}) in dimension $N-1\geq 2$ ensures us that the value of the last sum in the above estimate is $1/2$, regardless of what $v_2$ is. Therefore, we have
\[
\sum_{i=1}^{N+1}\scal{v}{x_i-x_1}^2 = \frac{N+1}2\, |v_1|^2 + \frac{1-|v_1|^2}2 = \frac{N|v_1|^2+1}2 \,,
\]
and the minimum of this expression among all $v\in \S^{N-1}$ is clearly $1/2$. Therefore, the proof is completed.
\end{proof}

We can now present our main results for the power-law kernel $g$ given by~(\ref{defprot}).

\begin{theorem}\label{thm:existence-power-like}
Let $N\geq 2$ and let $g=g_2$ be defined by~(\ref{defprot}), with $\alpha>\beta\geq 2$, $\alpha\geq 4$ and $(\alpha,\beta)\neq (4,2)$. Then, if $m$ is small enough, every minimizer of~(\ref{eq:problem-with-densities}) is the characteristic function of some set $E_m$ which is then a minimizer of~(\ref{eq:problem-with-sets}). Moreover, $E_m$ consists of $N+1$ convex components, each of which is contained in a small neighborhood of a vertex of $\Delta_N$.
\end{theorem}
\begin{proof}
With this choice of powers $\alpha,\,\beta$, we know by~\cite[Theorem~1.2,\, Corollary~1.4]{DLM2} that the measure $\mu_N$ defined in~\eqref{eq:simplex} is the only minimizer of $\E$ in $\P(\R^N)$, and that $\psi_{\mu_N}>\E(\mu_N)$ outside of $\spt \mu_N=\Delta_N$. We now want to compute the first and second derivatives of the function
\[
\psi_{\mu_N}=\frac 1{N+1}\, \sum_{i=1}^{N+1} \psi_{\delta_{x_i}}
\]
at the point $x_1$. First of all, notice that for a radial function $g:\R^N\to\R$ and any choice of $x\in\R^N,\, v\in\S^{N-1},\, t>0$ one has
\begin{equation}\label{eq:directional-derivatives}\begin{split}
\frac{d}{dt}\,g(x+tv) &= g'(|x+tv|)\,\frac{\scal{x}{v}+t}{|x+tv|}\,,\\
\restr{\frac{d^2}{dt^2}}{t=0}g(x+tv)  &= g''(|x|)\,\frac{\scal{x}{v}^2}{|x|^2}+\frac{g'(|x|)}{|x|}\,\left(1-\frac{\scal{x}{v}^2}{|x|^2}\right)\,.
\end{split}\end{equation}
So, with the function $g=g_2$, keeping in mind that $g'(1)=0$ and $g''(1)=\alpha-\beta$, we have for every $i\geq 2$ that
\[
\begin{split}
\partial^2_v \psi_{\delta_{x_i}}(x_1) &=g''(|x_i-x_1|)\,\frac{\scal{x_i-x_1}v^2}{|x_i-x_1|^2} + \frac{g'(|x_i-x_1|)}{|x_i-x_1|}\left(1-\frac{\scal{x_i-x_1}{v}^2}{|x_i-x_1|^2}\right)\\
&= \big(\alpha-\beta\big) \scal{x_i-x_1}v^2\,,
\end{split}
\]
while of course
\[
\partial^2_v \psi_{\delta_{x_1}}(x_1) = g''(0)\,.
\]
We have now to distinguish the cases $\beta=2$ and $\beta>2$. If $\beta>2$, then $g''(0)=0$ and then by Lemma~\ref{lemma:lower-bound-second-derivative}
\[
\partial^2_v \psi_{\mu_N}(x_1) = \frac 1{N+1}\sum_{i=2}^{N+1} \partial^2_v \psi_{\delta_{x_i}}(x_1)\geq \frac{(\alpha-\beta) K_N}{N+1}\geq \frac{\alpha-\beta}{2(N+1)}\,,
\]
so $\partial^2_v \psi_{\mu_N}(x_1)>0$ for every $v\in\S^{N-1}$. Instead, if $\beta=2$, then $g''(0)=-1$, and then
\begin{equation}\label{noticelater}
\partial^2_v \psi_{\mu_N}(x_1) = \frac 1{N+1}\bigg(-1 + \sum_{i=2}^{N+1} \partial^2_v \psi_{\delta_{x_i}}(x_1)\bigg) \geq \frac{-1+(\alpha-2)K_N}{N+1}\,.
\end{equation}
Since $K_N=1/2$ by Lemma~\ref{lemma:lower-bound-second-derivative} and $\alpha>4$ because we are considering $\beta=2$, then $-1+(\alpha-2)K_N>0$, hence again $\partial^2_v \psi_{\mu_N}>0$ for every $v\in\S^{N-1}$. The fact that any minimizer of problem~\eqref{eq:problem-with-densities} with $\norm{f_m}_1=m$ is given by a characteristic function $f_m=\Chi{E_m}$ is then ensured by Corollary~\ref{cor:existence-sets}. Moreover, we know that the sets $E_m$ converge to $\Delta_N$ in the Hausdorff sense by~(\ref{incluconv}) of Lemma~\ref{lemma:convergence}, and $m^{-1}D^2\psi_{f_m}$ converges to $D^2\psi_{\mu_N}$ locally uniformly as noticed in Theorem~\ref{thm:existence-sets}. As a consequence, $D^2\psi_{f_m}$ is strictly positive definite in a neighborhood of each point $x_i$ when $m$ is sufficiently small, and so the set $E_m\cap B_{1/2}(x_i)$ is convex for each $i$ because it coincides with the sublevel set of a convex function.
\end{proof}

\begin{remark}
The same result is true also if $N=1$ for $\alpha>\beta\geq 2,\, \alpha\geq 3$ and $(\alpha,\beta)\neq (3,2)$. The proof is exactly the same, the only difference is that the term in~(\ref{noticelater}) was strictly positive since $\alpha-2>2$ and $K_N=1/2$, while now it is strictly positive since $\alpha-2>1$ and $K_N=1$.
\end{remark}

\bigskip

In contrast, the next theorem shows that for certain choices of the parameters $\alpha$ and $\beta$ the minimizer of~(\ref{eq:problem-with-densities}) is the characteristic function of a set \emph{for all} values of $m$.

\begin{theorem}\label{thm:existence-radial}
Let $g$ be defined by~(\ref{defprot}) with $\beta=2$, $N\geq 2$ and $\alpha\in (2,4)$, or $\beta=2$, $N=1$ and $\alpha\in(3,4)$. Then, for every $m>0$ the minimizer $f_m$ of~(\ref{eq:problem-with-densities}) is the characteristic function of a radial set, which is either an annulus or a disk. In particular, as $m\searrow 0$, the set is an annulus which converges to a circle in Hausdorff distance.
\end{theorem}
\begin{proof}
Let us consider any $m>0$. Since $\alpha\in(2,4)$ and $\beta=2$, it is known that the energy $\E$ is strictly convex among the functions with barycenter in the origin, see for instance~\cite[Section~4]{CP}. This implies that there is only a single minimizer among the functions with barycenter in the origin, and thanks to the invariance of the energy by rotation we obtain that this minimal function has to be spherically symmetric. Since $f_m$ is spherically symmetric, and since $\alpha>2$ for $N\geq 2$ or $\alpha>3$ for $N=1$, \cite[Theorem~2.2]{DLM1} ensures that the potential $\psi_{f_m}$ has positive third derivative, that is, calling $\varphi(s)= \psi_{f_m}(s{\rm e}_1)$ one has $\varphi'''(s)>0$ for every $s>0$. Moreover, $\varphi'(0)=0$ because $\psi_{f_m}$ is regular and radial. This implies that all level sets of $\varphi$ are given by either one or two points, hence for every $\lambda\in\R$ the set $\{x\in\R^N:\, \varphi(|x|)=\lambda\}$ is negligible. Since Proposition~\ref{prop:EL} ensures that in the set where $0<f_m<1$ the potential has a constant value $\lambda$, this gives that the set $\{0<f_m<1\}$ is negligible, and this precisely means that $f_m$ is the characteristic function of some set $E_m$, which, in turn, is radial because so is $f_m$. Moreover, calling $I\subseteq \R$ the set such that $E_m= \{ x\in\R^N:\, |x|\in I\}$, again Proposition~\ref{prop:EL} ensures that $I=\{s\in\R:\, \varphi(s)<\lambda\}$ for some $\lambda\in\R$. Keeping again in mind that $\varphi'(0)=0$ and $\varphi'''(s)>0$ for all $s>0$, we have that the sublevel sets of $\varphi$ are all intervals, either of the form $(a,b)$ for some $0\leq a<b$, or of the form $[0,b)$ for some $b>0$. This means that $E_m$ is either an annulus or a disk. In particular, Lemma~\ref{lemma:convergence} ensures that $E_m$ is an annulus for $m\ll 1$, since it must converge in the Hausdorff sense to a circle for $m\searrow 0$. On the other hand, $E_m$ is surely a whole disk for $m\gg 1$.
\end{proof}

\bigskip

\subsection{Last results for general kernels}

We have seen that, by Theorem~\ref{thmDLM}, in the special case when $g$ is a power-law kernel of the form~(\ref{defprot}) for a suitable choice of the parameters $\alpha,\,\beta$, the unique minimizing measure is the purely atomic measure $\bar\mu$ uniformly distributed over the vertices of the regular $(N+1)$-gon $\Delta_N$. The goal of this last section is to show that minimality of such a measure does not necessarily require the particular form~(\ref{defprot}), but it can also be a consequence of more geometrical, general properties of $g$. Let us be more precise. If we assume, just to fix the ideas, that $g(0)=0,\, g(1)=-1$ and $g(t)>-1$ for every $t\neq 1$, then pairs of points in the support of an optimal measure have convenience to stay at distance $1$, but it is impossible that all pair of points have distance $1$ since every point of the support has distance $0$ from itself. It is reasonable to guess that in some cases the most convenient choice could be to have as many points as possible with mutual distance $1$, hence, with $N+1$ points in the vertices of a unit regular $(N+1)$-gon. In particular, one can imagine that this could happen whenever $g\approx -1$ only in a small neighborhood of $-1$, and $g$ is flat enough close to $0$. In this section, we are going to prove that it is indeed so. In order to present a simple proof of geometric flavour, we use highly non-sharp assumptions, and we write the proof in the planar case $N=2$ for simplicity of notations. The general case $N\geq 3$ does not require any different ideas. The only caveat is notational complication due to several indices. The final Remark~\ref{finalremark} discusses the case of higher dimensions with slight improvements of the constants.

The first result we present is a ``confinement result'', which says that if $g$ is close to $-1$ only close to $1$ and not so much negative otherwise, then an optimal measure must be supported in a union of $3$ small balls around the vertices of $\Delta_2$.

\begin{lemma}[Confinement around $\Delta_2$]\label{confinementlemma}
Let $g\in L^1_{\rm loc}(\R^2)$ be a radial, continuous function such that $g(0)=0$, $g(1)=\min g=-1$, and for some $\eta<1/64$ and $\xi<1/165$ one has
\begin{align*}
g(t)>-\eta \quad \hbox{for $t\in [0,3/2]\setminus (1-\xi,1+\xi)$}\,, && g(t)>0 \quad \hbox{for $t\geq 3/2$}\,.
\end{align*}
Then every minimizing measure $\mu$ is concentrated in the union of three sets with diameter less than $5\xi$ and mutual distance between $1-6\xi$ and $1+\xi$. More precisely, given any point in any of the three sets, its distance with each of the other two sets is between $1-6\xi$ and $1+\xi$.
\end{lemma}
\begin{proof}
The assumptions on $g$ imply that its graph must be in the shaded region in Figure~\ref{Fig2}, left (a possible choice of $g$ is depicted just as an example). Let $\mu$ be an optimal measure. We divide the proof in few steps.
\step{I}{The diameter of $\spt\mu$ is at most $3/2$.}
Let us call $\bar\mu$ the measure which is uniformly distributed over the vertices of an equilateral triangle of side $1$. Then by minimality of $\mu$ we have
\begin{equation}\label{easyesti}
\E(\mu) \leq \E(\bar\mu) = -\frac 23\,.
\end{equation}
Assume now the existence of $x_1,\, x_2\in\spt\mu$ with $|x_1-x_2|>3/2$. For a small $r\ll 1$, that will be specified in few lines, we can take two measures $\mu_1,\, \mu_2 \leq \mu$ so that $\rho:=\|\mu_1\|=\|\mu_2\|>0$ and 
$\spt\mu_i\subset B_{r/2}(x_i)$. For every $-1<\eps<1$ we define $\mu_\eps = \mu + \eps (\mu_1-\mu_2)$, which is still a probability measure. We have
\[
\E(\mu_\eps) = \E(\mu) + 2 \eps \Big( \E(\mu,\mu_1)-\E(\mu,\mu_2)\Big)+ \eps^2 \Big( \E(\mu_1)+\E(\mu_2)-2\E(\mu_1,\,\mu_2)\Big)\,.
\]
However, keeping in mind Proposition~\ref{prop:EL} and the fact that $\mu_1\leq \mu$, we have
\[
\E(\mu,\mu_1) = \int \int g(x-y) \, d\mu_1(x) \, d\mu(y) = \int \psi_\mu (x) \,d\mu_1(x) =  \rho \,\E(\mu) \,,
\]
and similarly $\E(\mu,\mu_2) = \rho \,\E(\mu) $. Therefore, the above expression becomes
\begin{equation}\label{termneg}
\E(\mu_\eps) = \E(\mu) + \eps^2 \Big( \E(\mu_1)+\E(\mu_2)-2\E(\mu_1,\,\mu_2)\Big)\,.
\end{equation}
Since by assumption $g(0)=0<C := g(|x_1-x_2|)/2>0$, we can pick $r>0$ so small that
\[
g(s)< C < g(t) \qquad \hbox{for every $0\leq s \leq r$ and $|x_1-x_2|-2r\leq t \leq |x_1-x_2|+2r$}\,.
\]
We have then
\[
\E(\mu_1) = \int \int g(y-x) \,d\mu_1\, d\mu_1 < C  \rho^2 \quad \text{ and } \quad
\E(\mu_2) = \int \int g(y-x) \,d\mu_2\, d\mu_2 < C  \rho^2\,,
\]
while
\[
\E(\mu_1,\, \mu_2) = \int \int g(y-x) \,d\mu_1\, d\mu_2  > C \rho^2\,.
\]
This ensures that the term in parentheses in~(\ref{termneg}) is strictly negative, giving $\E(\mu_\eps)<\E(\mu)$ which contradicts the minimality of $\mu$. This concludes the step.

\step{II}{The sets $A_x,\, A_y$ and $Q_{x,y}$.}
Let us now fix any point $x\in \spt\mu$, and call
\[
A_x = \big\{ y:\, 1-\xi < |y-x|<1+\xi\big\}\
\]
the annulus centered at $x$ with radii $1-\xi$ and $1+\xi$. By~(\ref{easyesti}) and minding~(\ref{eq:EL}), we have
\[\begin{split}
-\frac 23 &\geq \E(\mu) = \psi_\mu(x) = \int_{A_x} g(y-x)\,d\mu(y) + \int_{\R^2\setminus A_x} g(y-x)\, d\mu(y)
\geq -\mu(A_x) -\eta(1-\mu(A_x))\\
&\geq -\mu(A_x)-\eta\,,
\end{split}\]
which can be rewritten as
\begin{equation}\label{>23}
\mu(A_x) \geq \frac 23 - \eta\,.
\end{equation}

\begin{figure}[htbp]
\begin{tikzpicture}[>=>>>,scale=1.5]
\fill[red!15!white] (0,1)--(0,-0.3)--(0.85,-0.3)--(0.85,-1)--(1.15,-1)--(1.15,-0.3)--(1.5,-0.3)--(1.5,0)--(2.3,0)--(2.3,1)--(0,1);
\fill (0,-0.3) circle (0.03);
\draw (0,-0.3) node[anchor=east] {$-\eta$};
\fill (0.85,0) circle (0.03);
\draw (0.6,0) node[anchor=south] {$1-\xi$};
\fill (1.15,0) circle (0.03);
\draw (1.4,0) node[anchor=south] {$1+\xi$};
\fill (0,-1) circle (0.03);
\draw (0,-1) node[anchor=east] {$-1$};
\draw[red, line width=1pt] (0,-0.3)--(0.85,-0.3)--(0.85,-1)--(1.15,-1)--(1.15,-0.3)--(1.5,-0.3)--(1.5,0)--(2.3,0);
\fill (1.5,0) circle (0.03);
\draw (1.5,0) node[anchor=north west] {$3/2$};
\draw[->, line width=0.8pt] (0,-1.5)--(0,1.2);
\draw[->, line width=0.8pt] (0,0)-- (2.5,0);
\fill (1,-1) circle (0.03);
\draw (0,0) .. controls (1.4,0.1) and (0.7,-1) .. (1,-1) .. controls (1.3,-1) and (0.7,0.4) .. (2,0.6) .. controls (2.3,0.6) and (2.2,1) .. (2.4,1.3);
\fill[white] (2.2,1) rectangle (2.5,1.3);
\fill[green!15!white] (7.5,-0.2) arc(0:360:1.5) -- (7.1,-0.2) arc(360:0:1.1);
\fill[green!15!white] (8.8,-0.2) arc(0:360:1.5) -- (8.4,-0.2) arc(360:0:1.1);
\fill[orange!65!white] (7.05,0.87) arc (45.5:64:1.5) arc(116:134.5:1.5) arc(76.5:54:1.1) arc(126:103.5:1.1);
\fill[orange!65!white] (7.05,-1.27) arc (-45.5:-64:1.5) arc(-116:-134.5:1.5) arc(-76.5:-54:1.1) arc(-126:-103.5:1.1);
\draw[line width=0.8pt] (6,-0.2) circle (1.5);
\draw[line width=0.8pt] (6,-0.2) circle (1.1);
\draw[line width=0.8pt] (7.3,-0.2) circle (1.5);
\draw[line width=0.8pt] (7.3,-0.2) circle (1.1);
\fill (6,-0.2) circle (0.03);
\draw (6,-0.2) node[anchor=north] {$x$};
\fill (7.3,-0.2) circle (0.03);
\draw (7.3,-0.2) node[anchor=north] {$y$};
\fill (6.75,0.9) circle (0.03);
\draw (6.75,0.9) node[anchor=east] {$z$};
\draw (9,0.9) node[anchor=east] {$A_y$};
\draw (4.3,0.9) node[anchor=west] {$A_x$};
\draw (6.65,1.3) node[anchor=south] {$Q_{x,y}$};
\end{tikzpicture}
\caption{Left: the graph of $g$ must be in the shaded region. Right: the points $x,\,y$ and $z$ and the sets $A_x,\,A_y$ and $Q_{x,y}$ in Step~II.}\label{Fig2}
\end{figure}
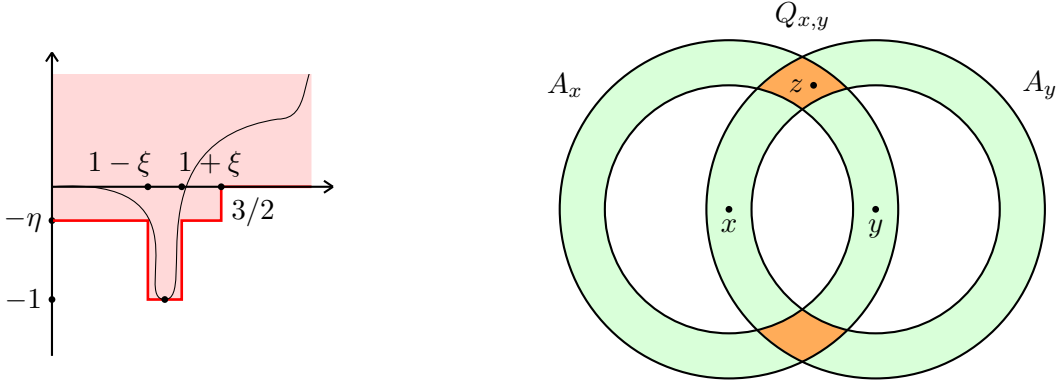

We can then take a second point $y$ in $\spt\mu$ so that $y\in A_x$, and then also $x\in A_y$. The intersection $A_x\cap A_y$ is made by two different connected pieces, orange in Figure~\ref{Fig2}, right. A trivial computation ensures that, by the assumption on $\xi$, the diameter of each piece is less than $5\xi$ and the distance between the two pieces is more than $3/2$. Step~I implies then that at least one connected piece of $A_x\cap A_y$ is $\mu$-negligible. On the other hand, applying~(\ref{>23}) both to $x$ and $y$ we obtain that $\mu(A_x\cap A_y)\geq 1/3-2\eta> 0$, and then exactly one connected piece of $A_x\cap A_y$ has positive $\mu$-measure. We call $Q_{x,y}$ this piece, so that, as just observed,
\begin{equation}\label{use0}
\mu(Q_{x,y}) \geq \frac 13 - 2\eta\,.
\end{equation}

\step{III}{The point $z$ and the conclusion.}
We can now define a third point $z\in \spt\mu\cap Q_{x,y}$, so that each of the annuli $A_x,\, A_y$ and $A_z$ centered at one of the points $x,\,y,\, z$ contains the other two points. Moreover, keeping the same notation as in Step~II, we call $Q_1=Q_{x,y}$, $Q_2=Q_{x,z}$ and $Q_3=Q_{y,z}$. Let now $w$ be any point in $\spt\mu$; since by Step~I we know that the distance between $w$ and any of the points $x,\,y,\,z$ is at most $3/2$, an immediate computation ensures that, thanks to the bound on $\xi$, the distance between $w$ and at least one of the points $x,\,y,\,z$ is less than $1-6\xi$. To fix the ideas, we can assume that
\begin{equation}\label{closetoz}
|z-w|< 1-6\xi\,.
\end{equation}
We assume then the existence of a point $v\in Q_1$ such that
\begin{equation}\label{contrass}
|v-w|>5\xi\,,
\end{equation}
and we look for a contradiction. Notice that this contradiction will conclude the proof; indeed, if (\ref{contrass}) is false for every $v\in Q_1$, there are some consequences. The first one is that the whole $\spt\mu$ is contained in the three balls of radius $5\xi$ centered at $x,\, y$ and $z$. Then, a second consequence is that the intersection of any of these balls with $\spt\mu$ has diameter at most $5\xi$, and thus $\mu$ is concentrated in the union of three sets with diameter less than $5\xi$. Moreover, by construction, for every point $a$ in one of these sets, the annulus $A_a$ intersects both the other two sets, and as a consequence the distance between $a$ and each of the other two sets is between $1-6\xi$ and $1+\xi$. Therefore, we only have to get a contradiction.\par

Let us write $\psi_\mu=\psi_1+\psi_2+\psi_3+\psi_\infty$, where we define
\begin{align*}
\psi_i(a)=\int_{Q_i} g(b-a)\,d\mu(b) \quad \forall\, i\in \{1,\,2,\,3\}\,, &&
\psi_\infty(a)=\int_{\R^2\setminus (Q_1\cup Q_2 \cup Q_3)} g(b-a)\, d\mu(b)\,.
\end{align*}
Notice now that, for every $b\in\spt \mu\cap Q_1$, since the diameter of $Q_1$ is less than $5\xi$ and by~(\ref{closetoz}) we have $|b-v|< 5\xi <1-\xi$ and $|b-w|<1-\xi$, and by the assumption on $g$ this implies
\begin{align}\label{use1}
\psi_1(w) > -\eta \mu(Q_1)\,, && \psi_1(v) > -\eta \mu(Q_1)\,.
\end{align}
Moreover, using~(\ref{>23}) also with $A_y$ and $A_z$ in place of $A_x$, and keeping in mind that by construction $Q_1\cap Q_2\cap Q_3=\emptyset$, we obtain that
\begin{align}\label{boundsmuQi}
\mu\big(\R^2\setminus(Q_1\cup Q_2\cup Q_3)\big)<  3\eta\,, && \mu(Q_1)\leq \frac 13 + \eta\,,
\end{align}
and the first bound immediately yields
\begin{equation}\label{use2}
\psi_\infty(w) + \psi_\infty(v) = \int_{\R^2\setminus (Q_1\cup Q_2 \cup Q_3)} g(b-w) + g(b-v)\, d\mu \geq -6\eta\,.
\end{equation}
Take now any two points $p\in Q_2,\, q\in Q_3$. As seen before, $A_p\cap A_q$ has diameter less than $5\xi$, so by the assumption~(\ref{contrass}) at least one between $v$ and $w$ does not belong to $A_p\cap A_q$, hence
\[
g(p-v)+g(p-w)+g(q-v)+g(q-w)\geq -3-\eta\,.
\]
Consequently, also by the second bound in~(\ref{boundsmuQi}) with $Q_2$ and $Q_3$ in place of $Q_1$, we have
\begin{equation}\label{use3}\begin{split}
\mu(Q_2)\big(\psi_3(w)&+\psi_3(v)\big) + \mu(Q_3)\big(\psi_2(w)+\psi_2(v)\big)\\
&=\int_{Q_2}\int_{Q_3} g(p-v)+g(p-w)+g(q-v)+g(q-w) \, d\mu(p)\,d\mu(q)\\
&\geq -(3+\eta) \mu(Q_2)\mu(Q_3)\geq -(3+\eta) \bigg(\frac 13+\eta\bigg)^2\,.
\end{split}\end{equation}
Again by~(\ref{easyesti}) and~(\ref{eq:EL}), and using~(\ref{use1}), (\ref{use2}), (\ref{use0}) and~(\ref{use3}) we have then
\[\begin{split}
-\frac 43 &\geq \psi_\mu(w)+\psi_\mu(v) = \psi_1(w)+\psi_2(w)+\psi_3(w)+\psi_\infty(w) +\psi_1(v)+\psi_2(v)+\psi_3(v)+\psi_\infty(v)\\
&> -2\eta \mu(Q_1)-6\eta + \bigg(\frac 13 - 2\eta\bigg)^{-1} \Big(\mu(Q_2)\big(\psi_3(w)+\psi_3(v)\big)+\mu(Q_3)\big(\psi_2(w)+\psi_2(v)\big)\Big)\\
&\geq -\frac 23\,\eta-2\eta^2 -6\eta -\bigg(\frac 13 - 2\eta\bigg)^{-1} (3+\eta) \bigg(\frac 13+\eta\bigg)^2\,,
\end{split}\]
and we derive the searched contradiction since this inequality is impossible for $\eta<1/64$.
\end{proof}

The main result of this section is that under suitable assumptions on the second derivative of $g$ around $0$ and around $-1$ the unique optimal measure is the purely atomic measure uniformly distributed over the vertices of a triangle of side $1$. More precisely, we have the following result:

\begin{theorem}\label{Thm3d}
Let $g$ be as in Lemma~\ref{confinementlemma}, and assume in addition that 
\[
g''(t)> -  12 \, g''(s) \qquad \forall\, t\in(0,5\xi),\, s\in(1-6\xi,1+6\xi)\,.
\]
Then, the unique optimal measure (up to translations and rotations) is the purely atomic one, uniformly distributed over the vertices of $\Delta_2$.
\end{theorem}
\begin{proof}
Let $\mu$ be an optimal measure. By Lemma~\ref{confinementlemma}, $\mu$ is concentrated on three sets $B_1,\, B_2,\, B_3$, with diameter less than $5\xi$ and mutual distance between $1-6\xi$ and $1+\xi$. Moreover, by~(\ref{use0}) and~(\ref{boundsmuQi}), each of them has measure between $\frac 13-2\eta$ and $\frac 13+\eta$. Let us call $C'=-\min \{g''(t) \colon 0\leq t \leq 5\xi\}$ and $C''=\min \{g''(t) \colon 1-6\xi\leq t \leq 1+6\xi\}$.\par

Let us take any four points $x,\,y,\, z,\, w$ in $\spt\mu$, in particular $x,\,y\in B_1$, $z\in B_2$ and $w\in B_3$. By construction and by Lemma~\ref{confinementlemma}, we have that
\begin{align*}
|y-x|\leq 5\xi\,, &&
|x-z|\leq 1+6\xi\,, &&
|x-w|\leq 1+6\xi \,, &&
|z-w|\geq 1-6\xi\,.
\end{align*}
Calling for brevity $\theta_{a,b}$ the direction of the vector $a-b$ for any two points $a\neq b\in\R^2$, the above estimates give
\begin{align}\label{generic}
\sin\big(|\theta_{x,z}-\theta_{y,z}|\big)\leq \frac{5\xi}{1-6\xi}\,, && \sin\bigg(\frac{|\theta_{x,z}-\theta_{x,w}|}2\bigg)\geq \frac 12\, \cdot\, \frac{1 -6\xi}{1+6\xi}\,.
\end{align}
Two elementary trigonometric estimates tell that, for a generic direction $\nu$,
\begin{equation}\label{deviation}\begin{split}
|\theta_{x,z} \cdot \nu|^2 + |\theta_{x,w}\cdot \nu|^2  &\geq 2 \sin^2\bigg( \frac {|\theta_{x,z}-\theta_{x,w}|}2\bigg)\,, \\ 
\Big||\theta_{x,z}\cdot \nu|^2 - |\theta_{y,z}\cdot \nu|^2\Big| &\leq \sin\big(|\theta_{x,z}-\theta_{y,z}|\big)\,.
\end{split}\end{equation}
In particular, we set $\nu=\theta_{x,y}$. Let us now consider the difference $\big||y-z|-|x-z|\big|$. By convexity of the distance, we can estimate this difference from below by $|y-x|$ multiplied either by $|\theta_{x,z}\cdot \nu|$ or by $|\theta_{y,z}\cdot \nu|$, unless the projection of $z$ onto the line passing through $x$ and $y$ is contained inside the segment $xy$, which means that $\theta_{x,z}$ and $\nu$ are very close to be perpendicular (and we discuss this case, which is in fact simpler, in a moment). We then have that
\[
\Big||y-z|-|x-z|\Big| \geq \min\Big\{ |\theta_{x,z}\cdot \nu| ,\,|\theta_{y,z}\cdot \nu| \Big\} \, |y-x|\,,
\]
which in turn yields
\begin{equation}\label{one}
g(x-z)+g(y-z) \geq -2 + \frac{C''}4\, \min\Big\{ |\theta_{x,z}\cdot \nu| ,\,|\theta_{y,z}\cdot \nu| \Big\}^2\, |y-x|^2 \,.
\end{equation}
We can now repeat the very same argument with $w$ in place of $z$.  Again unless $\theta_{x,w}$ is very close to be perpendicular to $\nu$, we have
\begin{equation}\label{two}
g(x-w)+g(y-w) \geq -2 + \frac{C''}4\, \min\Big\{ |\theta_{x,w}\cdot \nu| ,\,|\theta_{y,w}\cdot \nu| \Big\}^2\, |y-x|^2 \,.
\end{equation}
Putting together~(\ref{generic}) and~(\ref{deviation}), and in particular observing that the second estimate in~(\ref{generic}) holds also with $y$ in place of $x$ since $x$ and $y$ are generic points in $B_1$, we get that
\[\begin{split}
\min\Big\{ &|\theta_{x,z}\cdot \nu| ,\,|\theta_{y,z}\cdot \nu| \Big\}^2  + \min\Big\{ |\theta_{x,w}\cdot \nu| ,\,|\theta_{y,w}\cdot \nu| \Big\}^2\\
&\geq 2 \sin^2\bigg( \frac {|\theta_{x,z}-\theta_{x,w}|}2\bigg)- \Big||\theta_{x,z}\cdot \nu|^2 - |\theta_{y,z}\cdot \nu|^2\Big|
\geq  \frac 12\, \bigg(\frac{1 -6\xi}{1+6\xi}\bigg)^2-\frac{5\xi}{1-6\xi}\geq \frac 25\,,
\end{split}\]
where the last estimate is true by the assumption in Lemma~\ref{confinementlemma} that $\xi<1/165$. This last estimate together with~(\ref{one}) and~(\ref{two}) gives
\begin{equation}\label{anycase}
g(x-z)+g(y-z) + g(x-w)+g(y-w) \geq -4+ \frac{C''}{10} \, |y-x|^2 \,.
\end{equation}

Recall that~(\ref{anycase}) holds under the assumption that $\nu$ is not very close to be perpendicular to either $\theta_{x,z}$ or $\theta_{x,w}$. However, if this is the case then an even stronger estimate holds; in fact, if for instance $\nu$ is almost perpendicular to $\theta_{x,w}$, then we simply have
\[
\Big||y-z|-|x-z|\Big| +\Big||y-w|-|x-w|\Big| \geq \Big||y-z|-|x-z|\Big|
\geq \min\Big\{ |\theta_{x,z}\cdot \nu| ,\,|\theta_{y,z}\cdot \nu| \Big\} \, |y-x|\,,
\]
and since the minimum is close to $\sqrt 3/2$ because the triangle $xyz$ is nearly equilateral, the resulting estimate is stronger than~(\ref{anycase}). Hence, the validity of~(\ref{anycase}) is established in any case. Concerning $g(y-x)$, on the other hand, we have
\begin{equation}\label{estiC'}
g(y-x)\geq - \frac{C'}2\, |y-x|^2\,.
\end{equation}

Let us write $\psi^-=\psi_{\mu\smallres (B_2\cup B_3)}$, that is, $\psi^-(a)=\int_{B_2\cup B_3} g(a-b)\,d\mu(b)$. Using~(\ref{anycase}), and recalling that $g(x-z)+g(y-z)$ and $g(x-w)+g(y-w)$ are both surely greater than $-2$, we obtain
\begin{equation}\label{stmu12}\begin{split}
\psi^-(x)&+\psi^-(y) = \int_{B_2} g(x-z)+g(y-z)\, d\mu(z)  + \int_{B_3} g(x-w)+g(y-w)\, d\mu(w) \\
&\geq -2 \mu(B_2\cup B_3) + \frac{C''}{10}\, |y-x|^2 \min\big\{ \mu(B_2),\, \mu(B_3)\big\}\\
&\geq -2 \mu(B_2\cup B_3) + \frac{C''}{10}\, |y-x|^2 \bigg( \frac 13 - 2\eta\bigg)
\geq -2 \mu(B_2\cup B_3) + \frac{C''}{34}\, |y-x|^2\,.
\end{split}\end{equation}
We now evaluate $\E(\mu\res B_1,\mu)$, which by~(\ref{eq:EL}) coincides with $\mu(B_1) \E(\mu)$. We have
\[
\E(\mu\res B_1,\mu) = \int_{B_1} \int_{B_1} g(y-x)\,d\mu(y)\,d\mu(x)+ \int_{B_1}\int_{B_2\cup B_3} g(y-x)\,d\mu(y)\,d\mu(x)=:  \E_1+ \E_2\,.
\]
By~(\ref{estiC'}), we get
\[
\E_1 \geq -\frac{C'}2\, \int_{B_1} \int_{B_1} |y-x|^2\,d\mu(y)\,d\mu(x)\,.
\]
Instead, concerning $\E_2$, by~(\ref{stmu12}) we have
\[\begin{split}
\E_2 &= \int_{B_1} \psi^-(x)\,d\mu(x) = \frac 1{2\mu(B_1)}\, \int_{B_1} \int_{B_1} \psi^-(x)+\psi^-(y)\,d\mu(x) \, d\mu(y)\\
&\geq \frac 1{2\mu(B_1)}\, \int_{B_1} \int_{B_1} -2 \mu(B_2\cup B_3) + \frac{C''}{34}\, |y-x|^2 \,d\mu(x) \, d\mu(y)\\
&=  -\mu(B_1)\mu(B_2\cup B_3) + \frac{C''}{68 \mu(B_1)}\, \int_{B_1} \int_{B_1}\, |y-x|^2 \,d\mu(x) \, d\mu(y)\, .
\end{split}\]
Now, the assumptions
imply that $C''\geq 12\, C' \geq 34 \mu(B_1) C'$. Hence, from the two estimates above we get that
\[
\E(\mu\res B_1,\mu) \geq -\mu(B_1)\mu(B_2\cup B_3)\,,
\]
with strict inequality unless $\mu\res B_1$ is concentrated in a single point. Since the same estimate clearly works with $\mu\res B_2$ and $\mu \res B_3$ in place of $\mu\res B_1$, calling $c_i=\mu(B_i)$ for $i\in \{1,\,2,\,3\}$ and keeping in mind that $c_1+c_2+c_3=1$, we get
\[\begin{split}
-\E(\mu) &\leq c_1(c_2+c_3) + c_2(c_1+c_3) + c_3(c_1+c_2) = c_1(1-c_1) + c_2(1-c_2)+c_3(1-c_3)\\
&= 1 - \big(c_1^2+c_2^2+c_3^2\big) \leq \frac 23\,.
\end{split}\]
Since we already noticed that $\E(\mu)\leq - \frac 23$, we finally deduce that necessarily $c_1=c_2=c_3=\frac 13$ and each of the three measures $\mu\res B_i$ is concentrated in a single point. In addition, all the distances between any two of these three points must be equal to $1$, as claimed.
\end{proof}

\begin{remark}\label{finalremark}
In the general case of dimension $N\geq 3$, one can perform the very same construction as in Lemma~\ref{confinementlemma} and Theorem~\ref{Thm3d}, and obtain the very same results. More precisely, there are explicitly computable constants $\bar \eta,\,\bar\xi,\, c_1$ and $c_2$, only depending on the dimension, such that the following holds. If $g\in L^1_{\rm loc}(\R^N)$ is a radial, continuous function such that $g(0)=0$, $g(1)=\min g=-1$, and for some $\eta<\bar\eta$ and $\xi<\bar \xi$ one has $g(t)>-\eta$ for $t\in [0,\sqrt 3 H_N]\setminus (1-\xi,1+\xi)$ and $g(t)>0$ for $t\geq\sqrt 3 H_N$, then every minimizing measure is concentrated over the union of $N+1$ sets with diameter less than $c_1 \xi$ and mutual distance between $1-(1+c_1) \xi$ and $1+\xi$. In addition, if $g''(t)\geq - c_2 g''(s)$ for every $t\in (0, c_1\xi)$ and $s\in \Big(1-(1+c_1)\xi, 1+(1+c_1)\xi\Big)$, then the unique optimal measure is the purely atomic one, uniformly distributed over the vertices of $\Delta_N$.
\end{remark}

\bigskip
\subsection*{Acknowledgments}
D.C. is member of the Istituto Nazionale di Alta Matematica (INdAM), Gruppo Nazionale per l'Analisi Matematica, la Probabilit\`a e le loro Applicazioni (GNAMPA), and is partially supported by the INdAM--GNAMPA 2023 Project \textit{Problemi variazionali per funzionali e operatori non-locali}, codice CUP\_E53\-C22\-001\-930\-001. D.C. wishes to thank Virginia Commonwealth University for the kind hospitality received during his visit.

I.T.'s research was partially supported by a Simons Collaboration grant 851065, an NSF grant DMS 2306962, and through a visit to Universit\`{a} di Pisa where this project was initiated. He is grateful for the support and hospitality shown during this visit.


\end{document}